\documentclass[12pt]{article}
\usepackage{amsfonts,latexsym,rawfonts,amsmath,amssymb,amsthm}
\usepackage{lscape}
\usepackage[cm]{fullpage}
\usepackage{amscd,times, float,rotating}
\usepackage{pb-diagram}
\usepackage{hyperref}

\numberwithin{equation}{section}

\usepackage{geometry}
\newcommand{\pd}[2]{\frac {\partial #1}{\partial #2}}
\newcommand{\al}{\alpha}
\newcommand{\bb}{\beta}
\newcommand{\la}{\lambda}

\newcommand{\oo}{\omega}
\newcommand{\Om}{\Omega}

\newcommand{\dd}{\delta}
\newcommand{\Na}{\nabla}

\def\ga{\gamma}

\newcommand{\ee}{\epsilon}

\newcommand{\Si}{\Sigma}
\newcommand{\Te}{\Theta}
\newcommand{\te}{\theta}

\newcommand{\beq}{\begin{equation}}
\newcommand{\eeq}{\end{equation}}
\newcommand{\beqs}{\begin{eqnarray*}}
\newcommand{\eeqs}{\end{eqnarray*}}
\newcommand{\beqn}{\begin{eqnarray}}
\newcommand{\eeqn}{\end{eqnarray}}
\newcommand{\beqa}{\begin{array}}
\newcommand{\eeqa}{\end{array}}

\def\td{\tilde}

\def\ka{{\kappa}}

\def\RR{{\mathbb R}}
\def\NN{{\mathbb N}}
\def\PP{{\mathbb P}}
\def\CC{{\mathbb C}}

\def\SS{{\mathbb S}}
\def\ri{\rightarrow}

\def\pbp{\sqrt{-1}\partial\bar\partial}

\def\vol{{\rm Vol}}

\def\cA{{\mathcal A}}
\def\cC{{\mathcal C}}

\def\cF{{\mathcal F}}

\def\cL{{\mathcal L}}

\def\cO{{\mathcal O}}

\newtheorem{prop}{Proposition}[section]
\newtheorem{theo}[prop]{Theorem}
\newtheorem{lem}[prop]{Lemma}
\newtheorem{claim}[prop]{Claim}
\newtheorem{cor}[prop]{Corollary}
\newtheorem{rem}[prop]{Remark}

\newtheorem{defi}[prop]{Definition}

\title{Extremal K\"ahler metrics and energy functionals on projective   bundles}

\author{Haozhao Li\footnote{Research supported in part by National Science Foundation
of China No. 11001080 and a startup funding from University of
Science and Technology of China.}}

\begin{document}
\bibliographystyle{plain}

\date{}

\maketitle

\tableofcontents

\section{Introduction}
In \cite{[Ca1]}, E. Calabi introduced the extremal K\"ahler metric
on a compact K\"ahler manifold, which is a critical point of the
Calabi functional.  The existence of extremal K\"ahler metrics is a
long standing difficult problem, which is closely related to some
stabilities conditions in algebraic geometry. In the special case of
projective bundles, it is showed in literatures (cf.
\cite{[Ca1]}\cite{[Hwang]}\cite{[TF1]}\cite{[ACGF]} etc.) that the
extremal metrics can be explicitly constructed and have many
interesting properties. However, there exists a K\"ahler manifold
which admits no extremal metrics in certain K\"ahler classes. Thus,
a natural question is whether there are extremal metrics with
singularities on such manifolds and how the energy functionals
behaves. In the present paper, using the construction of
\cite{[ACGF]} we will study the relation between the existence of
extremal K\"ahler metrics and energy functionals on projective bundles.\\

The extremal metrics with conical singularities are studied on
Riemann surfaces in \cite{[Chen1]}\cite{[WZ]}. Similar to the smooth
case, it is believed that the existence of conical extremal metrics
is related to the behavior of energy functionals as well as some
stability conditions, as discussed by Donaldson in
\cite{[Do2]}\cite{[Do3]}.
 Based on the construction of extremal metrics on
projective bundles in \cite{[ACGF]}, we have the result:\\

\begin{theo}\label{theo1.1}On an admissible K\"ahler manifold $M
=\PP(\cO\oplus \cL)\ri S$, there exists a polymonimal $G_x(z)$ in
$z$  such that if $G_x(z)$ is positive on $(-1, 1)$ for some $x\in
(0, 1)$, then $M$ admits a conical extremal metric with
``sufficiently large" angle in the admissible K\"ahler class
corresponding to $x.$

\end{theo}

The notations in Theorem \ref{theo1.1} will be given in Section 2.
Theorem \ref{theo1.1} gives a criterion to determine whether there
exist conical admissible K\"ahler metrics on an admissible
manifolds. Moreover, following the same arguments in \cite{[ACGF]}
we can show that the existence of conical admissible extremal
metrics is equivalent to the positivity of a polynomial. In
\cite{[TF1]}, T{\o}nnesen-Friedman gave an interesting example which
admits no extremal metrics in some admissible K\"ahler classes.
However, using the arguments of Theorem \ref{theo1.1} we can show
that it admits
conical extremal metrics in any admissible K\"ahler class.\\

\begin{cor}\label{cor}On the admissible manifold $\PP(\cO\oplus \cL)\ri
\Si$ where $\Si$ is a Riemann surface with genus $g(\Si)>1$, there
exists a conical extremal metric
 in any admissible K\"ahler class.\\
\end{cor}

Next we will study the relations between conical extremal metrics
and energy functionals. Recall that in the smooth case, G. Tian
conjectured in \cite{[tian]} that the existence of extremal K\"ahler
metrics is equivalent to the properness of the modified $K$-energy,
which is a generalization of Mabuchi's $K$-energy by Guan
\cite{[Guan]} and Simanca \cite{[simanca]}.  In \cite{[ACGF]}, using
the theory of Chen-Tian \cite{[CT]}  a sufficient and necessary
condition is given for the existence of general extremal metrics on
an admissible manifold. Their results can be extended to conical
admissible extremal metrics except the auguments using Chen-Tian's
results. However, if we only consider the {\it admissible} K\"ahler
metrics, we can show the following  result:\\

\begin{theo}\label{theo02}Let $M=\PP(\cO\oplus \cL)\ri S$ be  an admissible
manifold. The following properties are equivalent for a conical
admissible K\"ahler class $\Om$:
\begin{enumerate}
  \item[(1)]  $M$  admits an extremal K\"ahler metric in
  $\Om$;
  \item[(2)] The extremal polynomial $F_{\Om}(z)$ is positive on $(-1, 1);$
  \item[(3)] The
modified $K$-energy is proper on $\Om.$\\
\end{enumerate}
\end{theo}

The admissible manifolds and the extremal polynomials were
introduced in \cite{[ACGF]}, and we will explain all the details in
Section 2. The equivalence of part (1) and part (2) of Theorem
\ref{theo02} is due to \cite{[ACGF]}.
 The proof on the properness of the modified $K$ energy relies on Donaldson
\cite{[Do1]} and Zhou-Zhu's work \cite{[ZZ]}, but we need to
carefully study the energy functionals in our situation. In the K\"ahler-Einstein
case, G. Tian prove the equivalence of the existence of K\"ahler-Einstein metrics and the
properness of the energy functionals in \cite{[Tian97]}.   \\

Now we study  the modified $K$-energy on admissible manifolds. The
lower boundedness of the modified $K$-energy is very subtle and it
is conjectured by X. X. Chen in  \cite{[Chen3]} \cite{[Chen4]} that
it is equivalent to the property that the infimum of the modified
Calabi energy is zero, and it might be related to the existence of
extremal metrics with singularities. On the admissible manifolds, we
can verify this conjecture and give the full criteria on the
modified $K$-energy
in terms of the extremal polynomial:\\

\begin{theo}\label{theo03}Let $M=\PP(\cO\oplus \cL)\ri S$ be  an admissible
manifold. Then the following properties are equivalent for a conical
admissible K\"ahler class $\Om:$
\begin{enumerate}

\item[(1)] The modified $K$-energy is bounded from below on $\Om$;
\item[(2)] The extremal polynomial $F_{\Om}(z)$ is nonnegative on $(-1, 1)$;
\item[(3)] The infimum of the modified Calabi energy on $\Om$ is zero.
\end{enumerate}
Moreover, if $F_{\Om}(z)$ is nonnegative and has $m$ distinct
repeated roots $z_i$ on $(-1, 1)$, then $M$ can split into $m+1$
parts, and each part admits an admissible extremal K\"ahler metric
with generalized cusp singularities at the ends $z=z_i.$

\end{theo}

 The generalized cusp singularity is defined in Section
\ref{Sec5.2}, and it is a generalization of the cusp singularity.
Combining Theorem \ref{theo03} with Theorem \ref{theo02}, we know
that the modified $K$-energy is bounded from below but not proper if
and only if the extremal polynomial is nonnegative and has repeated
roots on $(-1, 1).$ The phenomena that $M$ may admit complete
extremal metrics on each parts is   similar to the result of G.
Sz\'ekelyhidi in \cite{[S]}, where he discussed the minimizers of
the Calabi energy. It is easy to find an admissible manifold such
that the extremal polynomial satisfies this property. For example,
we check T{\o}nnesen-Friedman's example as in Corollary \ref{cor} and have the following:\\

\begin{cor}\label{cor02} On the admissible manifold $M=\PP(\cO\oplus \cL)\ri
\Si$ where $\Si$ is a Riemann surface with genus $g(\Si)>1$, there
is a point $x_s\in (0, 1)$ such that  for the admissible K\"ahler
class $\Om(x, 1)$ with $x\in (0, 1)$,
\begin{enumerate}
  \item[(1)] if $x\in (0, x_s),$ then $M$ admits a smooth admissible extremal
  metric on $\Om(x, 1);$
  \item[(2)] if $x=x_s,$ then the modified
$K$-energy is bounded from below but not proper on $\Om(x, 1). $ $M$
can split into two parts, and each part admits an admissible
extremal metric with  a cusp singularity on the fibre;
  \item[(3)] if $x\in (x_s, 1),$ $M$ can split into three parts,
  two of which has positive extremal polynomials and  admit   admissible extremal metrics with  conical
  singularities   on the fibre,  and one has negative extremal polynomial which
 determines no admissible extremal metrics with singularities.
\end{enumerate}

\end{cor}

The above results give close relations between the modified
$K$-energy and the existence of the extremal metrics. In a general
admissible manifold, the set of all admissible K\"ahler classes can
be divided into two subsets:  one  admits extremal metrics and the
other doesn't. The boundary K\"ahler classes of the two subsets have
the property that the modified $K$-energy is bounded from below but
not proper. We expect that these properties can be extended to toric
manifolds, and we will explore this in a forthcoming paper.\\

\noindent {\bf Acknowledgements}:  The author would like to thank
Professor Xiuxiong Chen and Xiaohua Zhu for warm encouragement and
stimulating discussions.

\section{Admissible K\"ahler metrics}
In this section, we recall some basic facts on the admissible
K\"ahler manifolds from \cite{[ACGF]}. The general admissible
K\"ahler manifolds are defined in \cite{[ACGF]} and here we only
consider a special case for simplicity.
\begin{defi}\label{defi:M} A projective vector bundle of the
form $M=\PP(\cO\oplus \cL)\ri S$ is called an admissible manifold if
$M$ satisfies the following properties.
\begin{enumerate}
  \item $S$ is a compact complex manifold covered by a product $\td
S=S_1\times S_2\times\cdots\times
  S_N$ of simply connected K\"ahler manifold $(S_i, g_i, \oo_i)$ of
  complex dimension $d_i$. Every metric $g_i$ has constant scalar curvature
  $S_{g_i}=2d_i s_i.$ $\cL$ denotes a holomorphic line bundle over
  $S.$
  \item $z$ is a Morse-Bott function on $M$  with image $[-1, 1]$ and the critical set
  $z^{-1}(\{-1, 1\}),$ and $M^0:=z^{-1}((-1, 1))$ is a principal
  $\CC^*$ bundle over $\td S.$
      \item  There are real numbers $x_i\in (0, 1), i=1,\cdots, N$
  such that the metric on $M^0$  is K\"ahler:
  \beqn
g&=&\sum_{i=1}^N\frac {1+x_iz}{x_i}g_i+\frac
{dz^2}{\Theta(z)}+\Theta(z)\theta^2;\label{eq0.1}\\
\oo_g&=&\sum_{i=1}^N\frac {1+x_iz}{x_i}\oo_i+dz\wedge
\te,\label{eq0.2}
  \eeqn where $\te$ is a connection $1$-form with $\theta(K)=1$ and $d\te=\sum_i\oo_i$.
  Here $K=J\Na_g z$ is a Killing vector field  generating the $\SS^1$ action on $M.$
  $\Te(z)$ is a smooth function on $[-1, 1]$ with
  \beq \Te(\pm1)=0,\quad \Te(z)>0,\;\;z\in (-1, 1)\eeq
and satisfies some addtional conditions which we will describe
below.\\
\end{enumerate}

\end{defi}

In \cite{[ACGF]}, the function $\Te(z)$ satisfies the boundary
conditions $\Te'(\pm 1)=\mp2$ so that the metric $g$ can extend to
$M$. In the present paper, we allow that each fibre of the
admissible manifold $M$ admits conical singularities.  Consider the
fiber metric \beq g_f=\frac
{dz^2}{\Theta(z)}+\Theta(z)\theta^2,\label{eq2.7}\eeq we define the
conical singularities below:
\begin{defi}\label{conical}A metric $g$ on a Riemann surface $\Si$ is called conical with angle
$2\pi\ka$ of order $\ga$ at a point $p\in \Si$ , if there is a
neighborhood $U$ of $p$ such that $g$ can be written in polar
coordinates as
$$g=ds^2+(\kappa^2s^2+O(s^{2+\ga}))\theta^2$$
for some $\kappa, \ga>0.$
\end{defi}
Note that in the definition \ref{conical} the metric $g$ is singular
at the point $p$ for $\ka\in (0, 1)$ and $g$ is degenerate at  $p$
for $\ka>1.$  Now we give some boundary conditions on $\Te(z)$ such
that each fibre has  conical singularities. For the purpose of
simplicity, we assume that each fibre has the singularities with the
same angle $2\pi \ka$ at $z=\pm 1.$ Define the set of functions for
$\ka>0$
$$\cA(\kappa)=\{\Te(z)\in C^{\infty}[-1, 1]\;|\;\Te(z)>0,\;z\in (-1, 1),\; \Te(\pm1)=0,\; \Te'(\pm1)=\mp2\kappa\}.$$
Note that  $\ka=1$ is exactly the smooth case discussed in
\cite{[ACGF]}.

\begin{lem}If $\Theta(z)\in \cA(\kappa)$ for some $\kappa>0$,
 then the fibre metric $g_f$ defined by (\ref{eq2.7}) has conic
 singularities with angle $2\pi \kappa$ of order $2.$
\end{lem}
\begin{proof}We only consider the neighborhood near $z=-1.$ Define
a function $s=s(z)$ by $$s(z)=\int_{-1}^z\;\frac
{dz}{\sqrt{\Te(z)}}.$$ Since $\Te(z)\in \cA(\ka),$ we can check that
$$\frac {d}{ds}\Te\Big|_{s=0}=\frac {d^3}{ds^3}\Te\Big|_{s=0}=0,\quad \frac {d^2}{ds^2}
\Te\Big|_{s=0}=2\ka^2,\quad \frac
{d^4}{ds^4}\Theta\Big|_{s=0}=4\ka^2\Te''(-1),$$ which implies that
$$g_f=ds^2+(\ka^2s^2+O(s^4))\theta^2.$$
The lemma is proved.\\
\end{proof}

The metric of the form (\ref{eq0.1}) for some smooth function
$\Te(z)\in \cA(\ka)$  is called a conical admissible K\"ahler metric
with angle $2\pi \ka$. The complex structure on the fibre will
change when the function $\Te(z)$ varies. However, after a
diffeomorphism every K\"ahler metric defined by different functions
$\Te(z)$ can be viewed as in the same K\"ahler class, which is
called conical admissible K\"ahler class and denoted by $\Om(x,
\ka)$. \\

We can calculate the scalar curvature of an admissible K\"ahler
metric.
\begin{lem}(cf. \cite{[ACGF]})\label{lem:scalar}
The scalar curvature of an admissible metric $g$ is given by
$$S_g=\sum_{i=1}^N\frac {2d_i s_i x_i}{1+x_iz}-\frac {F''(z)}{p_c(z)},$$
where $p_c(z)=\Pi_{i=1}^N (1+x_iz)^{d_i}$ and $F(z)=\Te(z)p_c(z).$

\end{lem}

The advantage of an admissible metric is that its scalar curvature
only depends on $z$. This directly implies that an admissible metric
is extremal if and only if the scalar curvature is an affine linear
function of $z.$ \\

Now we look for a function $\Te(z)\in \cA( \kappa)$ such that the
corresponding admissible metric $g$ is extremal  with the scalar
curvature $S_g+Az+B=0$ for some constants $A$ and $B.$ For any
$\Te(z)\in \cA(\ka)$, the function $F(z)=\Te(z)\,p_c(z)$ must
satisfy the conditions \beq F(\pm1)=0,\quad F'(-1)=2\kappa
p_c(-1),\quad F'(1)=-2\kappa p_c(1),\label{eq2.8}\eeq and $ F (z)>0$
on $(-1, 1). $ To construct   admissible extremal metrics, we define

\begin{defi}(cf. \cite{[ACGF]})\label{lem:F}  For an admissible K\"ahler
class $\Om(x, \ka),$ the extremal polynomial $F_{\Om}(z)$ is the
function satisfying $F_{\Om}(\pm 1)=0$ and  \beq
F_{\Om}''(z)=\Big(Az+B+\sum_i\frac {2d_is_ix_i}{1+x_iz}\Big)\cdot
p_c(z),\quad z\in (-1, 1). \label{eq2.5}\eeq Here the constants $A$
and $B$ are given by \beq A\al_1+B{\al_0}=-2\bb_{0, \ka},\quad
A\al_2+B\al_1=-2\bb_{1, \ka},\label{eq2.04}\eeq where $\al_r$ and
$\bb_{r, \ka}$ are defined by \beqn
\al_r&=&\int_{-1}^{1}\;p_c(t)t^r\,dt \label{eq2.01}\\
\bb_{r,\ka}&=&\kappa p_c(1)+(-1)^r\kappa
p_c(-1)+\int_{-1}^1\;\sum_i\frac
{d_is_ix_i}{1+x_it}p_c(t)t^r\,dt.\label{eq2.02} \eeqn\\
\end{defi}

 As in Proposition 8 of \cite{[ACGF]}, there is a unique polynomials
$F_{\Om}(z)$ satisfying the conditions in Definition \ref{lem:F}.
Moreover, by the uniqueness of $F_{\Om}(z)$,  we have the following
existence result:

\begin{theo}\label{theo2.1}
On an admissible K\"ahler manifold $M$,  there is an admissible
extremal K\"ahler metric with angle $2\pi \ka$ in an admissible
K\"ahler class $\Om(x, \ka)$ if and only if $F_{\Om}(z)$ is positive
on $(-1, 1)$.

\end{theo}

The proof of Theorem \ref{theo2.1} is  the same as in the case
$\ka=1$ of Proposition 8 in \cite{[ACGF]} and we omit it here. In
fact, using Chen-Tian's results of \cite{[CT]}, the result in
\cite{[ACGF]} says that the existence of a general extremal metric
in an admissible K\"ahler class is equivalent to the positivity of
the extremal polynomial on $(-1, 1)$. Thus, we would like to ask
whether the conical version of Chen-Tian's results hold and whether
we can generalize all the results in \cite{[ACGF]} to the conical
case.

\section{Existence of conical extremal metrics}

In this section, we will show a sufficient condition for the
existence of conical admissible  extremal metrics, and give an
example which admits no smooth
 extremal metrics in some admissible K\"ahler classes, but
does admit conical
extremal metrics in any admissible K\"ahler classes.\\

Following the arguments in \cite{[ACGF]}, we have the result:

\begin{theo}\label{theo3}On an admissible K\"ahler manifold $M$, there exists a polymonimal
$G_x(z)$ in $z$ which depends only on the function $p_c(z)$ such
that if $G_x(z)$ is positive on $(-1, 1)$ for some $x\in (0, 1)$,
then $M$ admits a conical extremal metric with "sufficiently large"
angle of order $2$ in the admissible K\"ahler class corresponding to
$x.$

\end{theo}
\begin{proof}Here we following the notations in Section 2. It suffices to find
when the extremal polynomial $F_{\Om}(z)$ is positive for $z\in (-1,
1).$  Note that (\ref{eq2.8}) implies
\beqn \int_{-1}^1\;F''_{\Om}(z)\,dz&=&-2\ka (p_c(1)+p_c(-1))\label{eq7.2}\\
\int_{-1}^1\;F''_{\Om}(z)z\,dz&=&-2\ka (p_c(1)-p_c(-1)).
\label{eq7.3} \eeqn Integrating (\ref{eq2.5}) and using
(\ref{eq7.2})-(\ref{eq7.3}), we have \beq A\al_1+B\al_0=-2\bb_{0,
\ka},\quad A\al_2+B\al_1=-2\bb_{1, \ka}, \label{eq7.4} \eeq where
$\al_r$ and $\bb_{r, \ka}$ are defined in Definition \ref{lem:F}.
Direct calculation shows that
$$A=\frac {2(\bb_{0, \ka}\al_1-\bb_{1, \ka}\al_0)}{\al_0\al_2-\al_1^2},\quad
B=\frac {2(\al_1\bb_{1, \ka}-\al_2\bb_{0,
\ka})}{\al_0\al_2-\al_1^2}.$$ Note that (\ref{eq2.5}) and
(\ref{eq7.2}) implies that
$$F_{\Om}(z)=2\ka p(-1)(z+1)+\int_{-1}^{z}\;\Big(At+B+\sum_{i=1}^N\;\frac {2d_is_ix_i}{1+x_it}\Big)p_c(t)(z-t)\,dt.$$
Observe that $F_{\Om}(z)$ is a linear function of $\ka,$ and we need
the coefficient of $\ka$ is positive for $z\in (-1, 1).$ The
coefficient of $\ka$ in the expression of $\frac
12(\al_0\al_2-\al_1^2)F_{\Om}(z)$ is \beqs
G_x(z):&=&(\al_0\al_2-\al_1^2)p(-1)(z+1)\\
&&+\Big((\al_1-\al_0)p_c(1)+(\al_1+\al_0)p_c(-1)\Big)\int_{-1}^z\;p_c(t)(z-t)t\,dt\\
&&+\Big((\al_1-\al_2)p_c(1)-(\al_1+\al_2)p_c(-1)\Big)\int_{-1}^z\;p_c(t)(z-t)\,dt,
\eeqs which depends only on the function $p_c(z).$ Since
$\al_0\al_2-\al_1^2>0$, $F_{\Om}(z)$ is positive for $z\in (-1, 1)$
if $G_x(z)>0(z\in (-1, 1))$ and $\ka$ is large
enough. The theorem is proved.\\

\end{proof}

The condition $G_x(z)>0(z\in (-1, 1))$ is less restrictive than the
positivity of the extremal polynomial, and it might be true for any
admissible class.  Here we discuss the example by C. T\o
nnesen-Friedman in \cite{[TF1]} where we can calculate
the angle $\ka$  explicitly.\\

\textbf{Example:}  \quad Let $\Si$ be a compact Riemann surface with
constant curvature metric $(g_{\Si}, \oo_{\Si}),$ and $M$ be
$P(\cO\oplus \cL)\ri \Si$ where $\cL$ is a holomorphic line bundle
such that $c_1(L)=\frac 1{2\pi}[\oo_{\Si}].$ Let $2s$ be the scalar
curvature of $g_{\Si}.$ By the Gauss-Bonnet theorem, we have
$$s=\frac {2(1-g(\Si))}{\deg \cL},$$
where $g(\Si)$ is the genus of $\Si.$ We consider the admissible
K\"ahler metrics of the form \beq g=\frac {1+xz}{x}g_{\Si}+\frac
{dz^2}{\Theta(z)}+\Theta(z)\theta^2,\quad x\in (0, 1), \eeq where
$\Te(z)\in \cA(\ka).$ By \cite{[TF1]} and \cite{[ACGF]}, if $s\geq
0$ then there exist extremal metrics in any admissible K\"ahler
classes. However, when $s<0$ there exist no extremal metrics in some
admissible K\"ahler class.
If we allow each fibre has conical singularities, then we have the result:\\

\begin{theo}On the admissible manifold $M$ with genus $g(\Si)>1$ as above,
for any $x\in (0, 1)$ there exists a conical extremal metric with
angle $2\pi \ka$ with
$$\ka >\frac {-sx^2}{(1-x)(3+x)}$$
 in the admissible K\"ahler class
corresponding to $x$.

\end{theo}
\begin{proof} Here $p_c(z)=1+xz$. We want to
find the extremal polynomial $F_{\Om}(z)=\Te_{\Om}(z)(1+xz)$ such
that $S_g+Az+B=0$ for two constants $A$ and $B.$ By Lemma
\ref{lem:scalar}, we have \beq
F_{\Om}''(z)=(1+xz)\Big(\frac{2sx}{1+xz}+Az+B\Big).\label{eq6.4}
\eeq Note that $F_{\Om}(z)$ satisfies the boundary conditions \beq
F_{\Om}(\pm1)=0,\quad F_{\Om}'(-1)=2\ka(1-x),\quad
F_{\Om}'(1)=-2\ka(1+x). \eeq Thus, $F_{\Om}''(z)$ satisfies \beq
\int_{-1}^1\;F''_{\Om}(z)dz=-4\kappa, \quad
\int_{-1}^1\;zF''_{\Om}(z)dz=-4\kappa x.\label{eq6.5} \eeq Combining
(\ref{eq6.4})-(\ref{eq6.5}), we have
$$A=\frac {6x(sx-2\ka)}{3-x^2},\quad B=\frac {6(\ka x^2-sx-\ka )}{3-x^2},$$
and the function $F_{\Om}(z)$ can be written as \beq
F_{\Om}(z)=\frac {(1-z^2)}{2(3-x^2)}\Big((2\ka x^2-sx^3)z^2+(6\ka
x-2\ka x^3)z+6\ka +sx^3-4\ka x^2\Big). \label{eq201}\eeq We want to
find when $F_{\Om}$ is positive for $z\in (-1, 1).$ Let \beqn
Q(z)&=&(2\ka x^2-sx^3)z^2+(6\ka x-2\ka x^3)z+6\ka +sx^3-4\ka
x^2\nonumber\\&=&
(6xz-2x^3z+6+2x^2z^2-4x^2)\ka-x^3sz^2+x^3s.\label{eq202}\eeqn Note
that the polynomial $6xz-2x^3z+6+2x^2z^2-4x^2$ is strictly positive
for all $x\in (0, 1)$ and $z\in (-1, 1)$, thus $M$ admits a conical
extremal metric in any admissible class. It is easy to find a
sharper bound of $\ka.$ In fact, we can check that
$$Q(1)=(6\ka -2\ka x^2)(1+x)>0,\quad Q(-1)=(6\ka -2\ka
x^2)(1-x)>0.$$ Thus, $Q(z)$ is positive on $(-1, 1)$ if the
following inequality holds
$$-\frac {6\ka x-2\ka x^3}{2(2\ka x^2-sx^3)}<-1,\quad \hbox{or}\quad \ka >\frac {-sx^2}{(1-x)(3+x)}.$$
The theorem is proved.\\

\end{proof}

\section{Estimates}
In this section, we will give some estimates on the modified
$K$-energy and $J$ functional which will be used in the proof of
main theorems.

\subsection{The symplectic potential}
For any $\Te(z)\in \cA(\ka),$ we define the symplectic potential
$u(z)$ of the admissible K\"ahler metric corresponding $\Te(z)$ by
\beq u''(z)=\frac 1{\Te(z)}.\label{eq2.1}\eeq Note that the
symplectic potential is unique up to an affine linear function.  Let
  $g_c$ be the admissible metric with its K\"ahler form $\oo_c$
defined by $\Te_{c, \ka}(z)=\ka(1-z^2)\in \cA(\ka),$ and we can
choose its symplectic potential to be
$$u_{c, \ka}(z)=\frac 1{2\ka}\Big((1-z)\log(1-z)+(1+z)\log(1+z)\Big).$$
 Denoted by $\cC_{\ka}$ the space of functions $u\in C^0([-1, 1])$
satisfying $u''(z)>0$ on $(-1, 1)$ and
$$  u-u_{c, \ka}\in
C^{\infty}([-1, 1]),\quad u(0)=u'(0)=0.$$We can check that for any
$u\in \cC_{\ka}$ the function $\frac 1{u''}$ belongs to $\cA(\ka),$
and thus it defines a conical admissible metric with angle $2\pi
\ka$.
\\

Now we relate the symplectic potential  to the K\"ahler potential.
For any symplectic potential $u\in \cC_{\ka},$ we define the
Legendre transform by \beq y=u'(z),\quad
\varphi(y)=-u(z)+yz,\label{eq2.2}\eeq where $z=(u')^{-1}(y)$ can be
viewed as a function of $y.$ We can check that $$
\varphi'_y(y)=z,\quad \varphi_{yy}''(y)=\Te(z)> 0,\quad z\in (-1,
1).$$ Here we denote $\varphi_y'(y)=d \varphi/d y$ and
$\varphi_{yy}''(y)=d^2\varphi/dy^2$ for simplicity. Note that the
complex structure defined by (\ref{eq0.1}) and (\ref{eq0.2}) on the
fibre is given by
$$J dz=\Te(z)\theta,\quad J\theta=-\frac 1{\Te(z)}dz,$$
we have $J dy=\theta$ and the equalities
$$dJd\varphi=dJ(\varphi'_y(y) dy)=d(z \theta)=z\sum_{i=1}^N\;\oo_i+dz\wedge \theta
=\oo_g-\sum_{i=1}^N \frac 1{x_i}\oo_i.$$ Now fix an admissible
K\"ahler form $\oo_{c}$ and its complex structure $J_{c}$, we have
the result:

\begin{lem}\label{lem:cs}(cf. \cite{[ACGF]}) There exists a fibre-preserving
 diffeomorphism $\Psi$ on
$M$ such that $\Psi^*J=J_c $ and $ \Psi^*y=y_c.$ Thus, any
admissible K\"ahler metric $\oo$ defined by $\Te(z)$ can be view as
in the same K\"ahler class  \beq \Psi^*\oo=\oo_c+dJ_c
d(\varphi(y_c)). \eeq Thus, the admissible K\"ahler class is
identified with the space $\cC_{\ka}.$
\end{lem}

\medskip
We denote by $\Om(x, \ka)$ the admissible K\"ahler class determined
by Lemma \ref{lem:cs} . Any metric in the K\"ahler class $\Om(x,
\ka)$ can be written as
$$\oo_g=\sum_{i=1}^N \frac 1{x_i}\oo_i+dJ_cd\varphi(y_c)=\oo_c+dJ_cd(\varphi(y_c)-\varphi_c(y_c)).$$

\subsection{The modified $K$-energy}
The modified $K$-energy is defined for extremal K\"ahler metrics by
Guan \cite{[Guan]} and Simanca \cite{[simanca]}, and it is a
generalization of the $K$-energy defined by Mabuchi  for
K\"ahler-Einstein metrics.

 Let $g$ be a K\"ahler metric on a compact K\"ahler
manifold $M$,
 $G$ be a maximal compact
connected subgroup of reduced automorphism group and $P_g$ the space
of Killing potentials with respect to any $G$-invariant metric $g$
in the admissible K\"ahler class $\Om.$ Define $pr_g$ the
$L^2$-projection to $P_g$. The modified $K$-energy is defined by
\beq\mu_{g_0}(\varphi)=-\int_0^1\;\int_M\;\pd
{\varphi_t}{t}pr_{g_t}^{\perp}S_{g_t}\;\oo_{g_t}^n\wedge
dt,\label{eq3.1}\eeq where $\varphi_t$ is a path in the space of
K\"ahler potentials which connects $0$ and $\varphi$ and
$\oo_{g_t}=\oo_0+\pbp \varphi_t.$ It can be shown that the
functional $\mu_{g_0}(\varphi)$ is independent of the choice of the
path $\varphi_t.$\\

\begin{lem}\label{lemR}(cf. \cite{[ACGF]}) If $g$ is an admissible metric
defined by $\Te(z)\in \cA(\ka)$, then the $L^2$ projection of $S_g$
orthogonal to the space of Killing potentials is
$$pr_g^{\perp}S_g=\frac {F_{\Om}''(z)-F''(z)}{p_c(z)},$$
where $F_{\Om}$ is the extremal polynomial of $\Om(x, \ka)$ and $F(z)=\Te(z)p_c(z).$\\
\end{lem}

For an admissible K\"ahler metric in $\Om(x, \ka)$, we still define
the modified $K$-energy by (\ref{eq3.1}). Note that for an
admissible metric, we have the volume form
$$dV_g=p_c(z)\Big(\bigwedge_{i=1}^N\frac {1}{d_i! x_i^{d_i}}\oo_i^{d_i}\Big)\wedge dz\wedge \te.$$
Using Lemma \ref{lemR} and integrating by parts, we  have \beqn
\mu_{g_c}(\varphi)&=&C_1\cdot \int_0^1\;dt\int_{-1}^1\;\pd ut
(F_{\Om}''(z)-F''(z))dz\nonumber\\
&=& C_1\cdot\int_0^1\;dt\int_{-1}^1\;\pd {u''}t (F_{\Om}(z)-F(z))dz\nonumber\\
&=&C_1\cdot\int_{-1}^1\;\Big(-p_c(z)\log\frac {u''(z)}{u_{c,
\ka}''(z)}+F_{\Om}(z)(u''(z)-u_{c, \ka}''(z))\Big)dz,
\label{eq:K}\eeqn where $C_1=2\pi \vol(S, \Pi_i\frac {\oo_i}{x_i})$
and we used the fact that $F(z)$ satisfies the same boundary conditions as $F_{\Om}(z).$ Thus, we have the lemma:\\

\begin{lem}
The modified $K$-energy $\mu_{g_c}(\varphi)$ is a positive multiple
of the functional \beq \cF (u)= -\int_{-1}^1\;p_c(z)\log\frac
{u''(z)}{u_{c, \ka}''(z)}dz+\int_{-1}^1\;F_{\Om}(z)(u''(z)-u''_{c,
\ka}(z))\; dz ,\nonumber \eeq where $u\in  \cC_{\ka}.$\\
\end{lem}

 It is proved
by Chen-Tian \cite{[CT]} that if a compact K\"ahler manifold admits
an extremal metric, then the modified $K$-energy is bounded from
below.  Following the argument in \cite{[ACGF]}, we can easily prove
if there is a conical  admissible extremal metric in $\Om(x, \ka)$,
then the modified $K$-energy is bounded from below in $\Om(x, \ka)$.
We will improve this result later.

\subsection{The $J$ functional}

In this section, we follow Zhou-Zhou \cite{[ZZ]} to discuss when the
$K$-energy is proper.  Recall that the $J$ functional defined by
Aubin on the space of K\"ahler potentials, \beq J_g(\varphi)=\frac
1V\int_0^1\int_M\;\pd {\varphi_t}t(\oo_g^n-\oo_{g_t}^n)\wedge
dt\label{eq3.03}\eeq where $\varphi_t$ is a path of K\"ahler
potentials connecting $0$ to $\varphi.$ As in the study of
K\"ahler-Einstein metric by Tian \cite{[Tian97]} , we introduce

\begin{defi}\label{defi:proper} The $K$-energy is called proper if
there is an increasing function $\rho(t)$ on $\RR$ with the property
that
$$\lim_{t\ri +\infty}\rho(t)=+\infty,$$
such that for any K\"ahler potential $\varphi,$
$$\mu_{\oo_g}(\varphi)\geq \rho (J_{ \oo_g}(\varphi)).$$\\

\end{defi}

 Recall that any function
$u\in  \cC_{\ka}$ can be written as $u=u_{c, \ka}+v$ for a smooth
function $v$ on $[-1, 1].$ The K\"ahler potential of $u$ and $u_{c,
\ka}$ is related by
$$\varphi(z)=-u(z)+u'(z)z,\quad \varphi_{c, \ka}(z)=-u_{c, \ka}(z)+u_{c, \ka}'(z)z.$$
Thus, the function $\td \varphi:=\varphi-\varphi_{c, \ka}$ is given
by \beq \td \varphi(z)=-v(z)+v'(z)z,\quad
\td\varphi'_z(z)=v_{zz}''z,\quad \td\varphi\in C^{\infty}[-1,
1].\label{eqb1}\eeq To estimate $J_{\oo_{g_c}}(\td \varphi)$ in the
admissible K\"ahler class, we have the result:

\begin{lem}\label{lemJ1}There exists a uniform constant $C$ such that for all
$u\in  \cC_{\ka}$ the corresponding $\td\varphi$ satisfies \beq
\Big|J_{\oo_{g_c}}(\td \varphi)-C_1\cdot
\int_{-1}^1\;u(z)\,dz\Big|\leq C,\label{eq403}
\eeq where $C_1=2\pi
\vol(S, \Pi_i\frac {\oo_i}{x_i}).$
\end{lem}
\begin{proof}We follow the argument of Zhou-Zhu \cite{[ZZ]} to prove
the lemma. By the definition (\ref{eq3.03}) of $J_{\oo_{g_c}}(\td
\varphi)$, we have \beqn J_{\oo_{g_c}}(\td\varphi)&=&\frac
1V\int_M\;\td\varphi\,\oo_{\oo_{g_c}}^n- \frac
1V\int_0^1\,\int_M\;\pd {\td\varphi}t\oo_{\varphi_t}^n\wedge
dt\nonumber\\&=& \frac 1V\int_M\;\td\varphi\,\oo_{\oo_{g_c}}^n+\frac
{2\pi}V\vol(S, \Pi_{i=1}^N \frac
{\oo_i}{x_i})\int_{-1}^1\;(u(z)-u_{c,
\ka}(z))p_c(z)dz.\label{eq402}\eeqn Thus, it suffices to show that
$\int_M\;\td\varphi\,\oo_{\oo_{g_c}}^n$ is uniformly bounded from
above and below.  \\

\begin{claim}\label{claim1} We have
\beq \td\varphi(z)\leq \frac
1V\int_M\;\td\varphi\,\oo_{\oo_{g_c}}^n+C, \label{eq4.01}\eeq
 for a uniform constant $C.$
\end{claim}

Note that (\ref{eq4.01}) is proved by the Green's function in
\cite{[ZZ]}, but we lack the lower bound of Green's function for
conical metrics here. However, we can prove it by direct
calculation.

\begin{proof}[Proof of Claim \ref{claim1}]
In fact, recall the fibre metric of $g_c$ is given by
$$g_{c, f}=\frac {dz^2}{\ka(1-z^2)}+\ka(1-z^2)\theta^2.$$
Since $\td\varphi$ is a K\"ahler potential and depends only on $z$,
its Laplacian satisfies \beq
\Delta_{\oo_{g_c}}\td\varphi=\Big((1-z^2)\td\varphi_z'\Big)_z'\geq
-C\label{eqa01}\eeq for a constant $C>0.$ Integrating (\ref{eqa01})
from $-1$ to $z$ and from $z$ to $1$ respectively, we have
$$-\frac {C}{1-z}\leq \varphi_z'\leq \frac C{1+z}.$$
Fix  $z_0\in [-1, 1]$, for any $z\geq z_0$ we have \beq
\td\varphi(z)-\td\varphi(z_0)=\int_{z_0}^z\,\td\varphi_z'(t)\,dt\geq
-\int_{z_0}^z\,\frac C{1-t}\,dt=C(\log(1-z)-\log(1-z_0)),
\nonumber\eeq and integrating $z$ from $z_0$ to $1$, we have \beq
(1-z_0)\td\varphi(z_0)\leq
\int_{z_0}^1\;\td\varphi(z)dz+C.\label{eqa02}\eeq On the other hand,
for $z\leq z_0$ we have \beq
\td\varphi(z_0)-\td\varphi(z)=\int_{z}^{z_0}\,\td\varphi_z'(t)\,dt\leq
\int_{z}^{z_0}\,\frac C{1+t}\,dt=C(\log(1+z_0)-\log(1+z)),
\nonumber\eeq and integrating from $-1$ to $z_0$ we have \beq
(z_0+1)\td\varphi(z_0)\leq
\int_{-1}^{z_0}\;\td\varphi(z)dz+C.\label{eqa03}\eeq Combining the
inequalities (\ref{eqa02})-(\ref{eqa03}) we have
$$\td\varphi(z_0)\leq \frac 12\int_{-1}^1\;\td\varphi(z)dz+C,$$
and the inequality (\ref{eq4.01}) is proved.\\

\end{proof}

Recall that the functions $\varphi$ and $\varphi_c$ defined by $u$
and $u_{c, \ka}$ respectively satisfy
$$y=u'(z),\quad y_c=u'_{c, \ka}(z).$$
Thus, $dy/dy_c>0$ and $y$ can be viewed as a function of $y_c$ for
all $y_c\in \RR.$ For this reason, we still denote by
$\varphi=\varphi(y_c)$ as a function of $y_c$.

\begin{claim}\label{claim2}
We have the inequality
$$\Big|\frac {d\varphi}{dy_c}\Big|\leq 1.$$

\end{claim}
\begin{proof}[Proof of Claim \ref{claim2}] Since $\td \varphi$ is a K\"ahler potential,
the function $\varphi=\varphi_c+\td \varphi$ is convex in $y_c$.
Thus, we have
$$\varphi(y_c)-\varphi(y_0)\geq z_0(y_c-y_0)$$
where $z_0=\frac {d\varphi}{dy_c}|_{y_c=y_0}$ for any $y_0\in \RR.$
Thus, the function $\varphi(y)-z_0y$ is bounded from below on $\RR.$
Direct calculation shows that
$$\varphi_c(y_c)=\log \frac {e^{y_c}+e^{-y_c}}{2},$$
and there is a uniform constant $C$ such that
$$\Big|\varphi_c(y_c)-|y_c|\Big|\leq C,\quad y_c\in \RR.$$
Therefore, for any $y_c\in \RR$ we have \beq |y_c|-z_0y_c\geq
\varphi_c(y_c)-z_0y_c-C\geq \varphi(y_c)-z_0y_c-C',\eeq which is
bounded from below. Here we used the fact that $\td
\varphi=\varphi(y_c)-\varphi_c(y_c)$ is a bounded function on $\RR.$
Since $|y_c|-z_0y_c$ is a piecewise linear function and bounded from
below, we have $z_0\in [-1, 1]$ and the lemma is proved.\\

\end{proof}

 Define the set
$$\Om_N=\{\xi\in M\;|\;\td\varphi(\xi)\leq \sup_M\td \varphi-N\}.$$
Thus, we can check that $\vol_{\oo_{g_c}}(\Om_N)\ri 0$ as $N\ri
+\infty.$ In fact, since \beqs \frac 1V\int_M\;\td \varphi\,
\oo_{\oo_{g_c}}^n&=&\frac 1V \int_{\Om_N}\;\td \varphi\,
\oo_{\oo_{g_c}}^n+\frac 1V \int_{M\backslash\Om_N}\;\td \varphi\,
\oo_{\oo_{g_c}}^n\\
&\leq&\frac {\vol(\Om_N)}{V}(\sup_M\td \varphi-N)+\frac
{V-\vol(\Om)}{V}\sup\td \varphi\\
&=&\sup_M\td \varphi-\frac {N\vol(\Om_N)}{V}.\eeqs Combining this
with the inequality (\ref{eq4.01}), we have
$$\vol(\Om_N)\leq \frac {CV}{N}\ri 0,\quad N\ri+\infty.$$
 On the other hand, since $\td \varphi$ satisfies
\beq \td \varphi(0)=0,\quad \Big|\frac {d\td\varphi}{dy_c}\Big|\leq
1,\label{eq401}\eeq we have
$$\td\varphi(y_c)=\td\varphi(y_c)-\td\varphi(0)\leq \sup_M\Big|\frac {d\td\varphi}{dy_c}\Big|\cdot |y|
\leq |y|.$$ Thus, for any $y_c\in (-1, 1)$ we have
$\td\varphi(y_c)\leq 1.$ Note that $\vol_g(\{p\in M\;|\; |y(p)|\leq
1\})$ is strictly positive, but the volume of the set $\Om_N\ri 0.$
Thus, there exists $y_0\in [-1, 1]$ such that
$$1\geq \td\varphi(y_0)\geq \sup_M\td\varphi(y)-N$$
for $N$ sufficiently large. Thus, $\sup_M\td\varphi(y)\leq N+1$ and
$\int_{M}\,\td \varphi\, \oo_{\oo_{g_c}}^n$ is bounded from above.
On the other hand, by Claim \ref{claim1} we have
$$\frac 1V\int_{M}\,\td \varphi\,
\oo_{\oo_{g_c}}^n\geq \td \varphi(0)-C=-C,$$ where we used
(\ref{eq401}). Since $\int_M\;\td \varphi\, \oo_{\oo_{g_c}}^n$ is
bounded from above and below, by (\ref{eq402}) we have the
inequality (\ref{eq403}). Thus, the lemma is proved.\\
\end{proof}

Define the operator $\cL$ on $\cC_{\ka}$ by
$$\cL u=\int_{-1}^1\;F_{\Om}(z)u''(z)
dz=\int_{-1}^1\;F_{\Om}''(z)u(z)
dz-F_{\Om}'(1)u(1)+F_{\Om}'(-1)u(-1) .$$ We have the result:

\begin{lem}\label{lemJ2} If there exists a constant $\dd>0$ such
that the inequality
$$\cL u\geq \dd \int_{-1}^1\;u(z)\,dz,$$
holds for any $u\in \cC_{\ka},$  then there exists a $\la>0$ such
that for any $u\in \cC_{\ka}$ we have
$$\cF(u)\geq \la \int_{-1}^1\;u(z)\,dz-C_{\la}.$$
\end{lem}
\begin{proof}
We choose a function $v_0\in \cC_{\ka}$ and define a function $G(z)$
by
$$G(z)=\frac {p_c(z)}{v_0''(z)}.$$
Thus, $v_0$ is a critical point of the functional \beqs \td
\cF(u)&=&\int_0^1\;dt\int_{-1}^1\;\pd ut
(G''(z)-F''(z))dz\\
&=& \int_0^1\;dt\int_{-1}^1\;\pd {u''}t (G(z)-F(z))dz\\
&=&\int_{-1}^1\;G(z)(u''(z)-u_{c,
\ka}''(z))dz-\int_{-1}^1\;p_c(z)\log\frac {u''(z)}{u_{c,
\ka}''(z)}dz, \eeqs which is a convex functional on $\cC_{\ka}.$
Thus, the functional $\td \cF(u)$ is bounded from below,
$$\td \cF(u)\geq \td \cF(v_0)\quad u\in \cC_{\ka}.$$
For any  positive constant $k>0$ and $u\in \cC_{\ka},$ we have \beqn
\td \cF(\frac 1k u)&=&\frac 1k \int_{-1}^1\;G(z)(u''(z)-ku_{c,
\ka}''(z))dz-\int_{-1}^1\;p_c(z)\log\frac
{u''(z)}{ku_{c, \ka}''(z)}dz\nonumber\\
&=&\frac 1k\td \cF(u)-\frac {k-1}{k}\int_{-1}^1\;G(z)u_{c,
\ka}''(z)dz+\log k\int_{-1}^1\;p_c(z)dz\nonumber\\&=&\frac 1k\td
\cF(u)-C_k\label{eq3.3}\eeqn for some constant $C_k.$ Thus, the
functional $\td \cF(\frac 1ku)$ is bounded from below on $\td
\cC_{\ka}.$\\

Define the functional \beq \td \cL u=\int_{-1}^1\;G(z)u''(z)\,dz
=\int_{-1}^1\;G''(z)u(z) dz-G'(1)u(1)+G'(-1)u(-1), \nonumber\eeq
where we used the fact that $G(z)$ satisfies the same boundary
conditions as $F_{\Om}(z)$.  Note that \beqs |\cL (u)-\td \cL
(u)|&=&\Big|\int_{-1}^1\;(G''(z)-F_{\Om}''(z))u(z)dz\Big|\\
&\leq& C\cdot \int_{-1}^1\;u(z)\,dz\leq \frac {C+\dd}{\dd}\cdot \cL
(u)-\dd \int_{-1}^1\;u(z)\,dz,\eeqs where $C$ is a positive constant
independent of $u(z).$ Thus, we get
$$\cL(u)\geq \frac {\dd}{C+2\dd}\td \cL(u)+\frac {\dd^2}{C+2\dd}\int_{-1}^1\;u(z)\,dz$$
and \beqs \cF(u)&\geq& \td \cF\Big(\frac {\dd
}{C+2\dd}\;u\Big)+\frac
{\dd^2}{C+2\dd}\int_{-1}^1\;u(z)\,dz-\log\frac {\dd}{C+2\dd} \\
&\geq&\frac {\dd^2}{C+2\dd}\int_{-1}^1\;u(z)\,dz-\log\frac
{\dd}{C+2\dd}-C', \eeqs where we used (\ref{eq3.3}) in the last
inequality. The lemma is proved.
\end{proof}

\section{Proof of main results}
In this section, we will use the estimates in Section 4 to prove
Theorem \ref{theo02}, Theorem \ref{theo03} and Corollary
\ref{cor02}.

\subsection{Proof of Theorem \ref{theo02}} In this
subsection, we will prove Theorem \ref{theo02}.  By Theorem
\ref{theo2.1}, it suffices to show that\\

\begin{theo}\label{theo001} On an admissible manifold $M$, there
exists an extremal K\"ahler metric on $\Om(x, \ka)$ if and only if
the modified $K$-energy is proper.

\end{theo}

\begin{proof}Suppose that $M$ admits an admissible extremal K\"ahler
metric on $\Om(x, \ka).$ To prove the properness of the $K$-energy,
by Lemma \ref{lemJ1} and Lemma \ref{lemJ2}
  it suffices to show that
  there is a $\dd>0$ such that for any $u\in
\cC_{\kappa},$ \beq \cL(u)\geq \dd\int_{-1}^1\,u(z)dz.
\label{eq004}\eeq
 In fact, since $F_{\Om}(z)$ is positive on $(-1, 1)$ and satisfies the boundary
 condition (\ref{eq2.8}), there is
a constant $c>0$ such that for any $z\in [0, 1]$ we have
$F_{\Om}(z)\geq c(1-z).$ Note that $u$ is convex, we have
$$\int_0^1\;F_{\Om}(z)u''(z)\,dz\geq c\int_0^1\;(1-z)u''(z)\,dz=c u(1).$$
Similarly, since $F'_{\Om}(-1)>0$ and $F(z)\geq c'(1+z)(z\in [-1,
0])$ for some constant $c'>0$, we have the inequality
$$\int_{-1}^0\;F_{\Om}(z)u''(z)\,dz\geq c'\int_{-1}^0\;(1+z)u''(z)\,dz=c' u(-1)$$
Combining the above inequalities and taking $\dd=\min\{c, c'\}$, we
have
$$\int_{-1}^1\;F_{\Om}(z)u''(z)\,dz\geq \dd(u(1)+u(-1))\geq 2\dd \int_{-1}^1\;u\,dz,$$
where we used the convexity of $u$ in the last inequality. Thus,
(\ref{eq004}) is proved and by Lemma \ref{lemJ1}-\ref{lemJ2} the modified $K$-energy
is proper.  \\

Now we show the necessity part of the theorem. Suppose that the
modified $K$-energy is proper. By Theorem \ref{theo2.1}, we only
need to show that the extremal polynomial $F_{\Om}(z)$ is positive
on $(-1, 1).$ Fix any $u\in \cC_{\ka}$. For any smooth nonnegative
convex function $f(z)$ on $[-1, 1]$ with $f(0)=0$ and $f'(0)=0$, the
functions $u_k=u+kf\in \cC_{\ka}$ for any $k\in \NN.$ We calculate
the modified $K$-energy of $u_k,$ \beqn \cF
(u_k)&=&-\int_{-1}^1\;p_c(z)\log\frac {u''_k(z)}{u_{c,
\ka}''(z)}dz+\cL(u_k)-\int_{-1}^1\;F_{\Om}(z)u''_{c, \ka}(z)\,dz\nonumber\\
&\leq &-\int_{-1}^1\;p_c(z)\log\frac {u''(z)}{u_{c,
\ka}''(z)}dz+\cL(u)+k\cL(f)-\int_{-1}^1\;F_{\Om}(z)u''_{c, \ka}(z)\nonumber\\
&=&\cF(u)+k\cL(f) .\label{eq002}\eeqn Using the inequality
(\ref{eq002}),  we have

\begin{claim}\label{claim3}
If the modified $K$-energy is bounded from below, then the extremal
polynomial $F_{\Om}(z)$ is nonnegative on $(-1, 1).$
\end{claim}
\begin{proof}The claim is due to \cite{[ACGF]} and we give
the details here for completeness.  Suppose  that $F_{\Om}(z)$ is
negative at some point on $(-1, 1).$ Then we can choose a
nonnegative smooth function $r(z)$ on $(-1, 1)$ such that
$$\int_{-1}^1\;F_{\Om}(z)r(z)dz<0.$$
Let $u_k$ be a sequence of functions in $\cC_{\ka}$ satisfying
$u_k''(z)=u_{c, \ka}''(z)+kr(z).$ As $k\ri+\infty$ we have
$$\cF(u_k)\leq \cF(u_{c, \ka})+k\int_{-1}^1\;F_{\Om}(z)r(z)dz\ri -\infty,$$
where we used (\ref{eq002}). Thus, the $K$-energy is not bounded
from below, a contradiction.\\

\end{proof}

Using Claim \ref{claim3}, we can construct a sequence of functions
with some special properties.

\begin{claim}\label{claim4}If $F_{\Om}(z)$ is nonnegative but not positive on $(-1, 1)$, then there is a sequence of smooth convex
functions $f_{k}(z)(k\in \NN)$ on $[-1, 1]$ with
 $f_{\ee}(0)=0, f_{\ee}'(0)=0$ and
 \beq \lim_{k\ri+\infty}\cL(kf_{k})=0,\quad
 \lim_{k\ri+\infty}\int_{-1}^1\;kf_{k}(z)dz=+\infty.\label{eq501}\eeq

\end{claim}

\begin{proof} Define the function $\eta(s)$ on by
\beq \eta(s)=\left\{
  \begin{array}{ll}
     e^{\frac 1{s^2-1}}, &   |s|\leq 1; \\
    0, &   |s|>1.
  \end{array}
\right.\label{eta}\eeq which is a smooth function on $\RR.$ Let
$h_{k}(s)=k\cdot \eta(k(s-z_0))$ and define
$$f_{k}(z)=\int_0^z\;(z-s)h_{k}(s)\,ds.$$
Then we can check that $f_{k}(z)$ is a smooth   convex function on
$(-1, 1)$ and satisfies $$f_{k}(0)=0,\quad f'_{k}(0)=0,\quad
f''_{k}(z)=h_{k}(z).$$ Note that for any $z_0\in (-1, 1)$ we have
$$\lim_{k\ri+\infty}\int_{-1}^1\;f_{k}(z)dz=\frac
12\int_{-1}^1\, (1-|z_0|) e^{\frac 1{t^2-1}}\,dt>0.$$ Thus, we have
$$\lim_{k\ri+\infty}\int_{-1}^1\;kf_{k}(z)dz=+\infty.$$\\

By Claim \ref{claim3} the extremal polynomial $F_{\Om}(z)$ is
nonnegative on $(-1, 1)$. If   $F_{\Om}(z)$ is not positive on $(-1,
1)$, it has repeated roots on $(-1, 1).$ Near a root $z_0\in (-1,
1),$ $F_{\Om}(z)$ can be expressed as $F_{\Om}(z)=g(z)(z-z_0)^{2m}$
for a positive function $g(z)$ on $[z_0-\ee, z_0+\ee]$ and an
integer $m\geq 1.$ Thus, we have
 \beqs
\cL(kf_{k})&=&\int_{-1}^1\;kF_{\Om}(z)f_{k}''(z)\,dz
=\int_{-1}^1\;k F_{\Om}(z)h_{k}(z)\,dz\\
&=&\int_{-1}^1\;k F_{\Om}(z_0+\frac tk)\;e^{\frac 1{t^2-1}}\;dt\\
&=&k^{1-2m}\int_{-1}^1\; g(z_0+\frac tk)t^{2m}\;e^{\frac
1{t^2-1}}\;dt\ri 0,\eeqs as $k\ri +\infty.$ The claim is proved.\\
\end{proof}

Now we proceed to prove the necessity part of the theorem. Since
$\cF(u)$ is proper,  there is an increasing function $\rho(t)$ such
that $\lim_{t\ri+\infty}\rho(t)=+\infty$ and \beq \cF(u)\geq \rho
\Big(\int_{-1}^1\,u(z)\,dz\Big),\quad \forall u\in
\cC_{\ka}.\label{eq500}\eeq If $F_{\Om}$ is not positive on $(-1,
1)$, by Claim \ref{claim4} we can construct $u_k=u+kf_k$ with the
property (\ref{eq501}). Combining this with the inequality
(\ref{eq500}) and (\ref{eq002}), we have
$$\rho \Big(\int_{-1}^1\,(u+kf_k)(z)\,dz\Big)\leq \cF(u)+\cL(kf_k)\ri \cF(u),$$
as $k\ri+\infty.$ However, the left hand side will tend to infinity,
which is a contradiction. Thus, $F_{\Om}(z)$ is positive on $(-1,
1)$ and the theorem is proved.
\end{proof}

\subsection{Proof of Theorem \ref{theo03}}\label{Sec5.2} In this section, we will
prove Theorem \ref{theo03}.  Theorem \ref{theo03} follows from the
following two results:

\begin{theo}On an admissible manifold $M$, the $K$-energy is bounded from
below if and only if $F_{\Om}(z)$ is nonnegative on $\Om(x, \ka).$
\end{theo}
\begin{proof}The necessity part is proved in Claim \ref{claim3}.
We only need to show the sufficiency part. Assume that $F_{\Om}$ is
nonnegative on $(-1, 1).$ If it is positive, then by Theorem
\ref{theo02} the $K$-energy is proper. Thus, it suffices to consider
the case when $F_{\Om}(z)$ has repeated roots on $(-1, 1).$ By the
expression of the $K$-energy,
$$\cF(u)=\int_{-1}^1\;\Big(-p_c(z)\log u''(z)+F(z)u''(z)\Big)\,dz+C.$$
Note that for any $a>0$, we have the inequality,
$$ax-\log x\geq 1+\log a,\quad x\in (0, \infty).$$
Thus, for any convex function $u,$ we have
$$-p_c(z)\log u''(z)+F_{\Om}(z)u''(z)\geq p_c(z)(1+\log \frac {F_{\Om}(z)}{p_c(z)}).$$
Since $p_c(z)$ is positive on $[-1, 1]$,  we only need to check
whether the integral \beq \int_{-1}^1\; \log F_{\Om}(z)\,dz>-\infty.
\label{eq003}\eeq In fact, near a root $z_0\in (-1, 1)$, the
polynomial can be expressed as $F_{\Om}(z)=g(z)(z-z_0)^{2m}$ for
some smooth function $g(z)$ which is positive on $[z_0-\ee,
z_0+\ee]$ and $m\in \NN.$ Here $\ee>0$ is sufficiently small such
that $F_{\Om}(z)$ has no other roots.  Note that
$$\int_{z_0-\ee}^{z_0+\ee}\;\log F_{\Om}(z)dz=
\int_{z_0-\ee}^{z_0+\ee}\;\log
g(z)dz+m\int_{z_0-\ee}^{z_0+\ee}\;\log (z-z_0)^2dz>-\infty.$$ Thus,
the inequality (\ref{eq003}) holds. The theorem is proved.\\

\end{proof}

Recall that  the modified Calabi energy can be expressed by
$$\int_M\;(pr^{\perp}S_g)^2\,dV_g=C\int_{-1}^1\;\frac {(F''(z)-F_{\Om}''(z))^2}{p_c(z)}\,dz$$
where we used Lemma \ref{lemR}. For simplicity, we define the
modified Calabi energy by $$Ca(u)=\int_{-1}^1\;\frac
{(q''(z))^2}{p_c(z)}\,dz,\quad q(z)=F(z)-F_{\Om}(z).$$ Now we have
the result:\\

\begin{theo}On an admissible manifold $M$, the infimum of the modified Calabi
energy on $\Om(x, \ka)$ is zero if and only if $F_{\Om}(z)$ is
nonnegative on $(-1, 1).$
\end{theo}
\begin{proof}Suppose that there is an interval $[a, b]\subset(-1,
1)$ such that $F_{\Om}(z)\leq -\ee $ is negative on $[a, b]$ for
some $\ee>0.$ Since for any $u\in \cC_{\ka}$ the function
$F(z)=\frac {p_c(z)}{u''(z)}$ is always positive on $(-1, 1),$ we
have \beq q(z)=F(z)-F_{\Om}(z)\geq \ee>0,\quad z\in [a,
b].\label{eq101}\eeq Note that $F(z)$ satisfies the same boundary
conditions (\ref{eq2.5}) as $F_{\Om}(z)$, the function $q(z)$ has
$$q(\pm 1)=0,\quad q'(\pm1)=0.$$
Therefore, we have the inequality
$$|q(z)|\leq \Big|\int_z^1\;q''(s)(s-z)\,ds\Big|\leq \frac 1{\sqrt{3}}
\Big(\int_z^1\;q''^2(s)\,ds\Big)^{\frac 12}(1-z)^{\frac 32}.$$
Combining this inequality with (\ref{eq101}), we have
$$Ca(u)\geq \frac 1{\la}\int_{-1}^1\;q''^2(s)\,ds\geq \frac {3\ee^2}{\la(1-a)^3},$$
where $\la=\max\{ p_c(z)|z\in [-1, 1]\}.$ Thus, the modified Calabi
energy has a positive lower bound. \\

Now we show the sufficiency part of the theorem. If $F_{\Om}(z)$ is
positive on $(-1, 1)$, then by Theorem \ref{theo02} $M$ admits
extremal metrics on $\Om(x, \ka)$ and hence the infimum of the
modified Calabi energy is zero. If $F_{\Om}(z)$ is nonnegative and
has repeated roots on $(-1, 1)$, we can choose a sequence of smooth
positive functions $F_n(z)$ with the boundary conditions
(\ref{eq2.5}) such that $F_n(z)$ converges smoothly to $F_{\Om}(z)$
on $(-1, 1)$. Then we can show that the modified Calabi energy
determined by $F_{n}(z)$ tends to zero. For example, suppose that
$F_{\Om}(z)$ has the only root $z_0\in (-1, 1)$ on $(-1, 1)$.  Then
we can choose a sequence of functions,
$$F_n(z)=F_{\Om}(z)+\frac 1{n^2}\eta(n(z-z_0)),$$
where $\eta$ is defined by (\ref{eta}), which is positive on $(-1,
1)$ and satisfies the same boundary as $F_{\Om}(z).$ Let $u_n(z)\in
\cC_{\ka}$ be the function determined by $F_n(z)=\frac
{p_c(z)}{u_n''(z)},$ we have \beqs Ca(u_n)&=&\int_{-1}^1\;\frac
1{p_c(z)}(F''_n(z)-F''_{\Om}(z))^2\,dz\\
&=&\frac 1n\int_{-1}^1\,\frac 1{p_c(z_0+\frac tn)}\,\frac
{4(3t^4-1)^2}{(t^2-1)^8}e^{\frac 2{t^2-1}}\,dt \ri 0.\eeqs
Similarly, we can show the infimum of the modified Calabi energy is
zero if $F_{\Om}(z)$ has many repeated roots on $(-1, 1)$. Thus, the
theorem is proved.\\
\end{proof}

Now we show the last part of Theorem \ref{theo03}. First we
introduce the singularities of a metric:

\begin{defi}\label{cusp} Let $\Si$  be a Riemann surface with
a metric $g$. A point $p\in \Si$ is called
\begin{enumerate}
  \item   a cusp point, if near the point $p$ the metric
$g$ can be written  as
$$g=\rho_1(s)ds^2+\rho_2(s)e^{-2s}\theta^2,\quad s\in [s_0, \infty),$$
where $\rho_1(s)$ and $\rho_2(s)$ are positive smooth functions at
$p.$
  \item   a generalized cusp point, if there is a integer $k\in \NN$ such that
near $p$ the metric  $g$ can be written  as
$$g=\rho_1(s)\frac {ds^2}{s^k}+\rho_2(s)s^k\,\theta^2,\quad s\in (0, s_0],$$
where $\rho_1(s)$ and $\rho_2(s)$ are positive smooth functions at
$p.$ In particular, if $k=2$  we can show that  $p$ is a cusp point.
\end{enumerate}

\end{defi}

 Suppose that
$F_{\Om}(z)$ is nonnegative and has distinct repeated roots
$z_i(1\leq i\leq m)$ with $-1<z_1<\cdots <z_m<1.$ Let $z_0=-1$ and
$z_{m+1}=1.$
 By Part 2 of Definition \ref{defi:M}, the manifold
$M_i=z^{-1}((z_i, z_{i+1}))$ for $0\leq i\leq m$ is a principal
$\CC^*$ bundle over $\td S$, and $F_{\Om}(z)$ is positive on $M_i$.
Thus, $M_i$ admits an admissible extremal K\"ahler metric and we
need to check the behavior of the metric near the ends $z=z_i, z_{i+1}$
if $z_i, z_{i+1}\neq \pm1:$\\

\begin{lem}\label{lem800}The admissible extremal metric on $M_i$ has
generalized cusp singularities at the ends $z=z_i, z_{i+1}$ if
$z_i, z_{i+1}\neq \pm1$.

\end{lem}
\begin{proof}Consider the fibre metric near $z=z_i, z_{i+1}(0\leq i\leq m),$
$$g_f=\frac {dz^2}{\Te_{
\Om}(z)}+\Te_{\Om}(z)\theta^2,\quad z\in (z_i, z_{i+1}).$$ Since
$z_i$ is a repeated root of $F_{\Om}(z)$, we can write
$F_{\Om}(z)=g(z)(z-z_i)^{2N}(N\in \NN)$ where $g(z)$ is positive on
$[z_i-\ee, z_i+\ee].$ Thus,   the fibre metric can be written as
\beq g_f=\frac {p_c(z)}{g(z)}\Big(\frac {dz^2}{(z-z_i)^{2N}}+\frac
{g(z)^2}{p_c(z)^2}(z-z_i)^{2N}\theta^2\Big),\quad z\in (z_i,
z_{i}+\ee). \label{eq800} \eeq In the special case of $N=1$, the
fibre metric (\ref{eq800}) has cusp singularities at $z=z_i$. In
fact, let $z-z_i=e^{-s}$ and we have
$$g_f=\rho_1(s)ds^2+\rho_2(s)e^{-2s}\te^2,$$
which is a metric with a cusp singularity. Here $\rho_1(s)$ and
$\rho_2(s)$ are smooth positive functions  near $z=z_i.$ For general
$N\in\NN,$ the metric (\ref{eq800}) has a generalized cusp
singularity, which is complete near $z_i.$\\
\end{proof}

\begin{rem}\label{rem}If $F_{\Om}(z)$ is negative at some points in $(-1, 1)$,
we define
$$I_+=\{z\in (-1, 1)\;|\;F_{\Om}(z)>0\},\quad I_-=\{z\in (-1, 1)\;|\;F_{\Om}(z)\leq 0\},$$
and $M_+=z^{-1}(I_+).$ Then $M_+$ has an admissible extremal metric
with singularities.

\end{rem}
In fact, near a boundary point of $M_+$ with $z=z_0\in \bar I_+\cap
\bar I_-$ and $(z_0, z_0+\ee)\subset I_+$ for small $\ee>0$, the
extremal polynomial can be written as
$F_{\Om}(z)=g(z)(z-z_0)^{2N-1}(N\in \NN)$ where $g(z)$ is smooth and
positive on $(z_0-\ee, z_0+\ee).$ We can discuss the singularities
at $z=z_0$ as in Lemma \ref{lem800}. For $N=1,$ the fibre metric at
$z=z_0$ has a conical singularity with angle $$2\pi \ka=\frac {\pi
F_{\Om}'(z_0)}{p_c(z_0)}>0.$$ For   $N\geq 2,$ the fibre metric has
generalized cusp singularities.

\subsection{Proof of Corollary \ref{cor02}} In this section, we will
show that after carefully choosing the parameters, the extremal
polynomial of the example in Section 3 is nonnegative and has a
repeated root on $(-1, 1)$, hence the modified
$K$-energy is bounded from below but not proper.\\

For simplicity, we only consider the smooth case $\ka=1$ of the
example in Section 3.  By the equality (\ref{eq201}) and
(\ref{eq202}), we need to find the parameters $x$ and $s$ such that
$$Q(z)=(2 x^2-sx^3)z^2+(6 x-2 x^3)z+6 +sx^3-4
x^2$$ has repeated roots on $(-1, 1).$ Note that since $s<0$ the
minimum point of $Q(z)$ \beq -\frac {6x-2x^3}{2(2 x^2-sx^3)}\in (-1,
1).\label{eq301}\eeq We would like to whether there is a root of
$$\Delta(x)=(6 x-2 x^3)^2-4(2 x^2-sx^3)(6 +sx^3-4x^2)=0.$$
In fact, we have\\

\begin{lem}There is a point $x_s\in (-1, 1)$ such that
  $\Delta(x)<0$ for $x\in (0, x_s)$ and     $\Delta(x)>0$ for
  $x\in (x_s, 1)$.

\end{lem}
\begin{proof}By direct calculation, we have
$$\Delta^{(4)}(x)=192+1440x^2-2880sx+1440s^2x^2>0,\quad x\in (0, 1).$$
Since $\Delta^{(3)}(0)=144s<0$ and
$\Delta^{(3)}(1)=672-1296s+480s^2>0,$ there is a point $x_3\in (0,
1)$ such that $\Delta^{(3)}(x)<0$ for $x\in (0, x_3)$ and
$\Delta^{(3)}(x)>0$ for $x\in (x_3, 1).$ Similarly, since
$\Delta^{''}(0)=-24<0$ and $\Delta^{''}(1)=192-336s+120s^2>0$, there
is a point $x_2\in (0, 1)$ such that $\Delta^{''}(x)<0$ for $x\in
(0, x_2)$ and $\Delta^{''}(x)>0$ for $x\in (x_2, 1).$ Now direct
calculation show that
$$\Delta'(0)=0,\quad \Delta'(1)=32-48s+24s^2>0.$$
We can also show that there is a point $x_1\in (0, 1)$ such that
$\Delta^{'}(x)<0$ for $x\in (0, x_1)$ and $\Delta^{'}(x)>0$ for
$x\in (x_1, 1).$ Combining this with $\Delta(0)=0, \Delta(1)=4+s^2>0
$, we know there is a point $x_s\in (0, 1)$ such that $\Delta(x)<0$
for $x\in (0, x_s)$ and $\Delta (x)>0$ for $x\in (x_s, 1).$

\end{proof}

 Thus, for any $s<0$, there is a $x_s\in (0, 1)$
such that $\Delta(x_s)=0$ and (\ref{eq301}) holds. For the
admissible K\"ahler class $\Om(x_s, 1)$, the extremal polynomial
$F_{\Om}(z)$ is nonnegative and has a repeated root $z_s\in (-1,
1).$ By Theorem \ref{theo02} and Theorem \ref{theo03}, the modified
$K$-energy is bounded from below but not proper. Moreover, $M$ can
split into two parts $M_0=z^{-1}((-1, z_s))$ and $M_1=z^{-1}((z_s,
1))$, and each part admits admissible extremal metrics with a cusp
singularity on the fibre.

If $x\in (0, x_s)$, $\Delta(x)<0$ and $Q(z)$ has no root on $(-1,
1).$ Thus, $M$ admits a smooth admissible extremal metric on $\Om(x,
1).$ If $x\in (x_s, 1)$, $\Delta(x)>0$. Note that
$$Q(1)=6+6x-2x^2-2x^3>0,\quad Q(-1)=(1-x)(6-2x^2)>0.$$
Combing this with (\ref{eq301}),  $Q(z)$ has   two simple zeros on
$(-1, 1).$ Thus, $M$ can split into
  three parts,  two of which satisfy $F_{\Om}(z)>0$ and admit
admissible extremal metrics with conical singularities on the fibre
by Remark \ref{rem}.

\noindent Department of Mathematics, \\
University of Science and Technology of China,\\
230026, Anhui province, China. \\
Email: hzli@ustc.edu.cn\\


\begin{thebibliography}{2}

\bibitem{[ACGF1]} V.~Apostolov, D.~M.~J. Calderbank and P.~Gauduchon,
Hamiltonian 2-forms in K\"ahler Geometry I:  General Theory, {\it J.
Differential Geom.} {  73} (2006), 359-412.

\bibitem{[ACGF2]} V.~Apostolov, D.~M.~J. Calderbank, P.~Gauduchon and
C.~T{\o}nnesen-Friedman,  Hamiltonian 2-forms in K\"ahler Geometry
II: Global Classification, {\it J. Differential Geom.} {  68}
(2004), 277-345.

\bibitem{[ACGF]} V.~Apostolov, D.~M.~J. Calderbank, P.~Gauduchon and
C.~T{\o}nnesen-Friedman, {  Hamiltonian 2-forms in K\"ahler
geometry, III. Extremal metrics and stability}, {\it Invent. Math.}
{  173} (2008), 547-601.

\bibitem{[BaMa]}S. Bando, T. Mabuchi,
 Uniqueness of Einstein K\"ahler metrics modulo connected group actions.
  Algebraic geometry, Sendai, 1985,  11-40, Adv. Stud. Pure Math., 10,
   North-Holland, Amsterdam, 1987.

\bibitem{[Ca1]}E. Calabi,  {  Extremal K{\"a}hler metrics}, Seminar on
Differential Geometry, Princeton Univ. Press (1982), 259-290.
\bibitem{[Ca2]}E. Calabi,  {  Extremal K{\"a}hler metrics} II, in:
Differential Geometry and Complex Analysis (eds. I. Chavel and H.M.
Farkas), Springer, Berlin, 1985.

\bibitem{[Chen1]}X. X. Chen, {  Extremal Hermitian metrics on Riemann surfaces},
{\it Calc. Var. Partial Differential Equations} 8 (1999), 191-232.

\bibitem{[CT]}X. X. Chen, G. Tian, {  Geometry of K\"ahler metrics and foliations
 by holomorphic discs}. {\it Publ. Math. Inst. Hautes \'etudes Sci.}
  No. 107 (2008), 1-107.

\bibitem{[Chen3]}X. X. Chen, {  Space of K\"ahler metrics. III.
 On the lower bound of the Calabi energy and geodesic distance }.
{\it Invent. Math.} 175 (2009), no. 3, 453-503.
\bibitem{[Chen4]}X. X. Chen, {  Space of K\"ahler metrics (IV)-
On the lower bound of the K-energy }, arXiv:0809.4081.


\bibitem{[Do1]}S. K. Donaldson, {  Scalar curvature and stability of
toric varieties}, {\it J. Differential Geom.} {  62} (2002),
289-349.
\bibitem{[Do2]}S. K. Donaldson,    Discussion of the K\"ahler-Einstein problem. Preprint.
\bibitem{[Do3]}S. K. Donaldson,   K\"ahler metrics with cone singularities along a
divisor,  arXiv:1102.1196.


\bibitem{[Guan]}D. Guan, { On modified Mabuchi functional and Mabuchi
moduli space of K\"ahler metrics on toric bundles}, {\it Math. Res.
Lett.} { 6} (1999), 547-555.


\bibitem{[Hwang]} A.D. Hwang,{  On existence of Kahler metrics with
constant scalar curvature}, {\it Osaka J. Math.} (1994), 561-595.




\bibitem{[S]}G. Szekelyhidi, {  The Calabi functional on a ruled surface}.
 {\it Ann. Sci. \'ec. Norm. Super.} (4) 42 (2009), no. 5, 837-856.

\bibitem{[simanca]} S. R. Simanca, { A K-energy characterization of
extremal K\"ahler metrics}, {\it Proc. Amer. Math. Soc.} { 128}
(2000), 1531-1535.

\bibitem{[TF1]} C. T\o nnesen-Friedman,
{  Extremal K\"ahler metrics on minimal ruled surfaces},{\it J.
reine angew. Math.} { 502} (1998), 175-197.

\bibitem{[tian]}G. Tian,  { Canonical metrics in K\"ahler geometry}.
Notes taken by Meike Akveld. Lectures in Mathematics ETH Z¨¹rich.
Birkh\"auser Verlag, Basel, 2000.

\bibitem{[Tian97]}G. Tian, K\"ahler-Einstein metrics with positive scalar curvature.
{\it Invent. Math.} 130 (1997), no. 1, 1-37.

\bibitem{[WZ]}G. Wang, X. Zhu, {  Extremal Hermitian metrics on Riemann
 surfaces with singularities}. {\it Duke Math. J.} 104 (2000), no. 2, 181-210.


\bibitem{[ZZ]}B. Zhou, X. H. Zhu  {  Relative K-stability and modified
K-energy on toric manifolds}, {\it Adv. Math.} {  219} (2008),
1327-1362.

\bibitem{[ZZ2]} B. Zhou, X. H. Zhu, K-stability on toric manifolds.
{\it Proc. Amer. Math. Soc.} 136 (2008), no. 9, 3301-3307.
\end{thebibliography}
\end{document}